\documentclass[11pt]{article}
\usepackage{mathrsfs}
\usepackage[intlimits]{amsmath}
\usepackage{amsfonts}
\usepackage{amssymb}
\usepackage{amsthm}
\usepackage{tikz}
\usepackage{blkarray}

\newtheorem{theorem}{Theorem}[section]
\newtheorem{definition}[theorem]{Definition}
\newtheorem{lemma}[theorem]{Lemma}

\newtheorem{cor}[theorem]{Corollary}
\newtheorem{notation}[theorem]{Notation}
\newtheorem{example}[theorem]{Example}
\newtheorem{remark}[theorem]{Remark}

\title{\Large\bf The endomorphism semiring of a commutative inverse semigroup}

\author {\textbf{M. K. Sen, S. K. Maity \& Sumanta Das} \\
\small\it Department of Pure Mathematics, University  of  Calcutta \\
\small\it 35, Ballygunge Circular Road, Kolkata-700019, India.\\
\small e-mail: senmk6@gmail.com \, \,  \, skmpm@caluniv.ac.in \, \,  \,  sumanta.das498@gmail.com}
\date{}

\begin{document}

\maketitle

\begin{abstract}
The authors \cite{jez} proved that the endomorphism semiring of a nontrivial semilattice is always subdirectly irreducible and described its monolith. Here we prove that the endomorphism semiring of a commutative inverse semigroup with at least two idempotents is always subdirectly irreducible and describe its monolith. 
\end{abstract}

\vspace{.5em}
AMS Mathematics Subject Classification (2010): 20M14, 20M18, 08B26, 16Y60.

{\bf Key Words:} Inverse semigroup, semilattice, endomorphism, subdirectly irreducible, monolith.

\section{Introduction}
A groupoid $(S,\cdot)$ is a semigroup if the binary operation `$\cdot$' is associative, i.e., $(a\cdot b)\cdot c = a\cdot (b \cdot c)$ for all $a, b, c \in S$. A semigroup $S$ is said to be an inverse semigroup if for each element $s \in S$, there exists a unique element $s' \in S$ such that $ss's = s$ and $s'ss' = s'$. The set of all idempotent elements in a semigroup $S$ is denoted by $E(S)$. 

A semiring $(S, +, \cdot)$ is an algebra with two binary operations `$+$' and `$\cdot$' such that the semigroup reducts $(S, +)$ and $(S, \cdot)$ are connected by distributive laws, viz., $a\cdot(b + c) = a\cdot b + a\cdot c$ and $(a+b)\cdot c = a\cdot c+b\cdot c$, for all $a, b, c \in S$. A semiring $S$ is called an additive inverse semiring if for each $a \in S$, there exists a unique element $a' \in S$ such that $a + a' + a = a$ and $a' + a + a' = a'$. Additive inverse semirings were first studied by Karvellas \cite{Karvellas} in 1974 and he proved that for any $a,b \in S$, $(a\cdot b)' = a'\cdot b = a\cdot b'$, $(a')' = a$ and $(a+ b)' = a' + b'$.  The set of all additive idempotents in a semiring $S$ is denoted by $E^+(S)$. 

An algebra $A$ is said to be subdirectly irreducible if it is nontrivial (i.e., has at least two elements) and intersection of arbitrary collection of non-identity congruences on it is again a non-identity congruence, i.e., it is nontrivial and among its non-identity congruences there exists the least one; this least non-identity congruence on $A$ is called the monolith of $A$. An algebra is called simple if it is nontrivial and has only two congruences (the identity congruence and the universal congruence). Through out this paper, we always denote the identity congruence on a semiring $S$ by ${\epsilon}_{_S}$ or by simply by 
$\epsilon$ when no confusion arises. 

The authors \cite{jez} proved that the endomorphism semiring of a nontrivial semilattice is always subdirectly irreducible and described the monolith of the endomorphism semiring of a nontrivial semilattice. They also proved that the endomorphism semiring is congruence simple if and only if the semilattice has both a least and a largest element.

For all undefined terms and definitions in semigroup theory we refer to \cite{howie} and \cite{Petrich}. 

\section{The Monolith}	

Let $(G, +)$ be a commutative inverse semigroup. Then $End(G)$, the set of all endomorphisms of $G$ forms an additive commutative as well as an additive inverse semiring, where addition is pointwise addition of two mappings and multiplication is composition of two mappings. For two elements $a, b\in G$, we define 
$a \, {\leq}_q \, b$ if and only if $a+a'+b=b$. It is easy to verify that the relation `${\leq}_q$' is a quasi order relation on $(G, +)$. In this connection, it is interesting to point out that for a commutative inverse semigroup $G$, the quasi order relation `${\leq}_q$' becomes a partial order relation on $E(G)$. For each element $a\in G$, we define $\lambda_a : G \longrightarrow G$ by $\lambda_a(x) = a+a'$, for all $x \in G$. Then it is easy to verify that for all $a \in G$, $\lambda_a \in End(G)$ satisfying the properties $\lambda_a+\lambda_b = \lambda_{a+b}, \, f \cdot \lambda_a = \lambda_{_{f(a)}}$ and $\lambda_a \cdot f = \lambda_a$ for all $f \in End(G)$. Thus $\textbf{L}_G = \{\lambda_a : a \in G \}$ is an ideal of $End(G)$. In this connection, it is worth mentioning that an element $f \in End(G)$ is a constant mapping if and only if $f = \lambda_a$ for some $a \in G$. 


For every triple $a, b, c$ of elements of $G$ with $a \, {\leq}_q \, b$, we define $\mu_{a,b,c} : G \longrightarrow G$ by : for all $x \in G$, \vspace{-.5em}
$$\mu_{a,b,c}(x)  = \left\{\begin{array}{l} a+a', \hspace{1.4em}
\mbox{ if}  \hspace{.4em}  x \, {\leq}_q \, c \\
b+b', \hspace{1.4em} \, \, \, \mbox{if} \hspace{.4em}   x \, {\nleq}_q \, c.
\end{array}\right.$$
We now show that $\mu_{a,b,c} \in End(G)$. For this, let $x, y$ be any two elements of $G$. Since $a \, {\leq}_q \, b$, we must have $a+a'+b=b$. 
 
\noindent $Case - 1:$ If $x+y \, {\leq}_q \, c$, then $x \, {\leq}_q \, x+y \, {\leq}_q \, c$ as well as $y \, {\leq}_q \, x+y \, {\leq}_q \, c$. Hence $\mu_{a,b,c}(x)+\mu_{a,b,c}(y) = a+a'+a+a'= a+a'= \mu_{a,b,c}(x+y)$.
 
\noindent $Case - 2:$ If $x+y \, {\nleq}_q \, c$, then at least one of $x \, {\nleq}_q \, c$ and $y \, {\nleq}_q \, c$ must hold. 

$Subcase - (A):$ Suppose $x \, {\leq}_q \, c$ and $y \, {\nleq}_q \, c$. Then $\mu_{a,b,c}(x)+\mu_{a,b,c}(y) = a+a'+b+b'=(a+a'+b)+b'=b+b' = \mu_{a,b,c}(x+y)$.
   	
$Subcase - (B):$ Suppose $x \, {\nleq}_q \, c$ and $y \,{\leq}_q \, c$. Then $\mu_{a,b,c}(x)+\mu_{a,b,c}(y) = b+b'+a+a' = b+b' = \mu_{a,b,c}(x+y)$.

$Subcase - (C):$ Suppose $x \, {\nleq}_q \, c$ and $y \, {\nleq}_q \, c$. Then $\mu_{a,b,c}(x)+\mu_{a,b,c}(y) = b+b'+b+b' = b+b'=\mu_{a,b,c}(x+y)$.
   
\noindent Considering all the cases, we have $\mu_{a,b,c} \in End(G)$. The subsemiring of $End(G)$ generated by $\mu_{a,b,c} \,$ (where $a, b, c \in G$ with $a \, {\leq}_q \, b$) is denoted by $\textbf{M}_G$. In this connection, we point out that $\textbf{M}_G$ is a left ideal of $End(G)$ and is an ideal of $End(G)$ if $E(G)$ is finite. 

\begin{notation}
Let $(G, +)$ be a commutative inverse semigroup. Then $G$ is a Clifford semigroup and hence $G$ is a semilattice of groups. For any $a \in G$, we denote the identity element of the subgroup of $G$ containing 
the element $a$ by $a^0$. Since $G$ is commutative, it follows easily that $a^0 = a+a' = a'+a$.
\end{notation}

\begin{remark} \label{rem:monolith}
Let $(G, +)$ be a commutative inverse semigroup. Then $E(G)$ is a semilattice and for any two elements $a, b \in G$, we have $a \, {\leq}_q \, b$ in $G$ if and only if $a^0 \leq b^0$ in $E(G)$, where $`\leq$' is the partial order on a semilattice. Again, for every triple $a, b, c$ of elements of $G$ with $a \, {\leq}_q \, b$, though the endomorphism $\mu_{a,b,c}$ on $G$ is different from the endomorphism ${\mu_{a,b,c}|}_{_{E(G)}}$ on the subsemilattice $E(G)$, but $\mu_{a,b,c}(x) = \mu_{a,b,c}(x^0) = \mu_{a^0,b^0,c^0}(x) = \mu_{a^0,b^0,c^0}(x^0)$ for all $x \in G$. Therefore, composition two $\mu$'s on $G$ is same as composition of corresponding two $\mu$'s on the semilattice $E(G)$. Thus from the proof of \cite[Theorem 3.4]{jez}, it follows that composition of two $\mu$'s is again a $\mu$. 
\end{remark}

\begin{definition}
A nonempty subset $A$ of $End(G)$ with at least two elements is said to be separated by idempotents if for any two $f, g \in A$ with $f \neq g$, there exists an element $e \in E(G)$ such that $f(e) \neq g(e)$.
\end{definition}

\begin{example}
We consider the Clifford semigroup $S_1 = (\mathbb{Z}_{3} , \cdot)$ and the semilattice $S_2 = (\left\{ 0 , 1 \right\},\cdot)$, where $\mathbb{Z}_{3}$ is the set of all residue classes of integers modulo $3$. Let $G$ be the direct product of $S_1$ and $S_2$. Then $G$ is a commutative inverse semigroup with $E(G) = \left\{(\overline{0} , 0), (\overline{0} , 1), (\overline{1} , 0), (\overline{1} , 1)\right\}$.
	
Then all the nine endomorphisms of $G$ are given by :
	
	(i) $f_1:G \rightarrow G$ defined by $f_1(x,y) = (\overline{0} , 0)$, for all $(x,y) \in G$.
	
	(ii) $f_2:G \rightarrow G$ defined by $f_2(x,y) = (\overline{0} , 1)$, for all $(x,y) \in G$.
	
	(iii) $f_3:G \rightarrow G$ defined by $f_3(x,y) = (\overline{1} , 0)$, for all $(x,y) \in G$.
	
	(iv) $f_4:G \rightarrow G$ defined by $f_4(x,y) = (\overline{1} , 1)$, for all $(x,y) \in G$.	
	
	(v) $f_5:G \rightarrow G$ defined by $f_5(x,y) = (\overline{0} , y)$, for all $(x,y) \in G$.	
	
	(vi) $f_6:G \rightarrow G$ defined by $f_6(x,y) = (\overline{1} , y)$, for all $(x,y) \in G$.	
	
	(vii) $f_7:G \rightarrow G$ defined by $f_7(x,y) = (x , 0)$, for all $(x,y) \in G$.
	
	(viii) $f_8:G \rightarrow G$ defined by $f_8(x,y) = (x , 1)$, for all $(x,y) \in G$.
	
	(ix) $f_9:G \rightarrow G$ defined by $f_9(x,y) = (x , y)$, for all $(x,y) \in G$.
	
\noindent It can be easily checked that any subset $A \subseteq End(G)$ with $|A|\geq 2$ is separated by idempotents.
\end{example}

\begin{remark} 
Let $G$ be a commutative inverse semigroup. Then $\mu_{a,a,c} = \lambda_a$ for any $c\in G$ and thus $\textbf{L}_G \subseteq \textbf{M}_G$.
\end{remark}

\begin{definition} \label{def:con1}
Let $(G, +)$ be a commutative inverse semigroup. We define a relation $\mathscr{R}_I$ on $End(G)$ by : for $f, g \in End(G)$,
\begin{center}
$f \, \, \mathscr{R}_I \, \, g$ if and only if $f+\lambda_a = g+\lambda_a$ for some $a \in G$.
\end{center}
\end{definition}

\begin{theorem}
Let $(G, +)$ be a commutative inverse semigroup. Then the relation $\mathscr{R}_I$, defined in Definition \ref{def:con1}, is a congruence on $End(G)$.	
\end{theorem}

\begin{proof}
Clearly, the relation $\mathscr{R}_I$ is reflexive and symmetric.
	
For transitivity, let $f, g, h \in End(G)$ such that $f \, \, \mathscr{R}_I \, \, g$ and $g \, \, \mathscr{R}_I \, \, h$. Then there exist elements $a, b \in G$ such that $f+\lambda_a = g+\lambda_a$ and $g+\lambda_b = h+\lambda_b$. Now, $f+\lambda_{a+b} = f+\lambda_a+\lambda_b = g+\lambda_a+\lambda_b = g+\lambda_b+\lambda_a = h+\lambda_b+\lambda_a = h+\lambda_a+\lambda_b =  h+\lambda_{a+b}$ and thus $\mathscr{R}_I$ is transitive. 

Clearly, $\mathscr{R}_I$ is a congruence on $(End(G), +)$. To show $\mathscr{R}_I$ is a congruence on $(End(G), \cdot)$, let $f, g, h \in End(G)$ such that $f \, \, \mathscr{R}_I \, \, g$. Then there exists an element $a \in G$ such that $f+\lambda_a = g+\lambda_a$. Now, $h \cdot f +\lambda_{_{h(a)}} = h \cdot f + h \cdot \lambda_a = h\cdot (f+\lambda_a) = h \cdot (g+\lambda_a) = h \cdot g + h \cdot \lambda_a = h \cdot g +\lambda_{_{h(a)}}$ implies $(h\cdot f) \,  \, \mathscr{R}_I \, \, (h\cdot g)$. Similarly, $(f\cdot h) \,  \, \mathscr{R}_I \, \, (g\cdot h)$. Hence $\mathscr{R}_I$ is a congruence on $(End(G), \cdot)$ and consequently, $\mathscr{R}_I$ is a congruence on the semiring $End(G)$. 
\end{proof}

\begin{remark}
Let $(G, +)$ be a commutative inverse semigroup containing at least two idempotents. Then $\mathscr{R}_I \neq {\epsilon}$ on $End(G)$.
\end{remark}

\begin{proof}
Since $E(G)$ contains at least two idempotent elements, we must have $e_1, e_2 \in E(G)$ such that $e_1 \neq e_2$. Then, clearly $\lambda_{e_1} \neq \lambda_{e_2}$. Now, $\lambda_{e_1}(x)+\lambda_{e_1+e_2}(x) = e_1+e_1'+(e_1+e_2)+(e_1+e_2)' = e_1+e_2 = \lambda_{e_2}(x)+\lambda_{e_1+e_2}(x)$, for all $x\in G$ implies $\lambda_{e_1}+\lambda_{e_1+e_2} = \lambda_{e_2}+\lambda_{e_1+e_2}$. Hence $(\lambda_{e_1}, \lambda_{e_2})\in \mathscr{R}_I$ and thus $\mathscr{R}_I \neq {\epsilon}$ on $End(G)$.
\end{proof}

\begin{definition} \label{def:con2}
Let $(G, +)$ be a commutative inverse semigroup. We define a relation $\mathscr{R}_L$ on $End(G)$ by : for $f, g \in End(G)$, $f \, \, \mathscr{R}_L \, \, g$ if and only if the following two conditions are satisfied : 

(i) the range of $f$ is a lower bounded subset of $G$, i.e., there exists $a\in G$ such that $a \, {\leq}_q \, \, f(x)$, for all $x\in G$,
		
(ii) the range of $g$ is a lower bounded subset of $G$, i.e., there exists $b\in G$ such that $b \, {\leq}_q \, \, g(x)$, for all $x\in G$.
\end{definition}

\begin{theorem}
Let $(G, +)$ be a commutative inverse semigroup. Then the relation $\mathscr{R}_L$, defined in Definition \ref{def:con2}, is a congruence on $End(G)$.	
\end{theorem}

\begin{proof}
Clearly, $\mathscr{R}_L$ is an equivalence relation on $End(G)$.
		
To show $\mathscr{R}_L$ is a congruence on $(End(G), +)$, let $(f , g)\in \mathscr{R}_L$ and $h \in End(G)$. Since $(f, g) \in \mathscr{R}_L$, there exist elements $a, b\in G$ such that $a \, {\leq}_q \, \, f(x)$ and $b\, {\leq}_q \, \, g(x)$, for all $x \in G$. Now $a \, {\leq}_q \, \, f(x)$ implies $a+a'+f(x)=f(x)$, i.e., $a+a'+f(x)+h(x) = f(x)+h(x)$, i.e., $a+a'+(f+h)(x) = (f+h)(x)$ for all $x \in G$. Hence $a \, {\leq}_q \, \, (f+h)(x)$ for all $x \in G$. Similarly, we can show that $b \, {\leq}_q \, \, (g+h)(x)$ for all $x\in G$. Therefore $(f+h, g+h)\in \mathscr{R}_L$ and hence $\mathscr{R}_L$ is a congruence on $(End(G), +)$.
		
Finally, to show $\mathscr{R}_L$ is a congruence on $(End(G), \cdot)$, let $(f , g)\in \mathscr{R}_L$ and $h \in End(G)$. Since $(f, g) \in \mathscr{R}_L$, there exist elements $a, b\in G$ such that $a \, {\leq}_q \, \, f(x)$ and $b \, {\leq}_q \, \, g(x)$, for all $x \in G$. Then $a \, {\leq}_q \, \, f(x)$ for all $x \in G$ implies $a \, {\leq}_q \, \, f(h(x))$ for all $x \in G$, i.e., $a \, {\leq}_q \, \, (f\cdot h)(x)$ for all $x \in G$. Similarly, it follows that $b \, {\leq}_q \, \, (g\cdot h)(x)$ for all $x\in G$. Therefore $(f\cdot h, g\cdot h)\in \mathscr{R}_L$ and hence $\mathscr{R}_L$ is a right congruence on $(End(G), \cdot)$. Again, since $a \, {\leq}_q \, \, f(x)$, for all $x \in G$, we have $a+a'+f(x) = f(x)$, for all $x \in G$. This implies $h(a+a'+f(x)) = h(f(x))$, for all $x \in G$, i.e., $h(a)+(h(a))'+(h\cdot f)(x) = (h\cdot f)(x)$, for all $x \in G$, i.e., $h(a) \, \, {\leq}_q \, \, (h\cdot f)(x)$ for all $x \in G$. Similarly, from $b \, {\leq}_q \, \, g(x)$, for all $x \in G$, we can show that $h(b) \, \, {\leq}_q \, \, (h\cdot g)(x)$, for all $x \in G$. Therefore $(h\cdot f, h\cdot g)\in \mathscr{R}_L$ and hence $\mathscr{R}_L$ is a left congruence on $(End(G), \cdot)$. Thus, $\mathscr{R}_L$ is a congruence on $(End(G), \cdot)$ and consequently, $\mathscr{R}_L$ is a congruence on the semiring $End(G)$. 
\end{proof}

\begin{remark}
Let $(G, +)$ be a commutative inverse semigroup containing at least two idempotent elements. Then $\mathscr{R}_L \neq {\epsilon}$ on $End(G)$.
\end{remark}

\begin{proof}
Since $E(G)$ contains at least two idempotent elements, we must have $e_1, e_2 \in E(G)$ such that $e_1 \neq e_2$. Clearly, $\lambda_{e_1} \neq \lambda_{e_2}$. Now $e_1+e_1'+\lambda_{e_1}(x) = \lambda_{e_1}(x)$ and $e_2+e_2'+\lambda_{e_2}(x) = \lambda_{e_2}(x)$, for all $x\in G$, and hence $e_1\leq \lambda_{e_1}(x)$, $e_2\leq \lambda_{e_2}(x)$, for all $x\in G$. Therefore $(\lambda_{e_1} , \lambda_{e_2})\in \mathscr{R}_L$ and  thus $\mathscr{R}_L \neq {\epsilon}$ on $End(G)$.
\end{proof}

\begin{theorem} \label{th:monolith}
Let $(G, +)$ be a commutative inverse semigroup containing at least two idempotents and $E$ be a subsemiring of $End(G)$ such that $\textbf{M}_G \subseteq E$ and $E$ is separated by idempotents. Then $E$ is subdirectly irreducible and its monolith is ${\mathscr{R}|}_{_E}$, where $\mathscr{R} = \mathscr{R}_I \cap \mathscr{R}_L$ and ${\mathscr{R}|}_{_E} = \mathscr{R} \cap (E\times E)$. 
\end{theorem}
	
\begin{proof}		
Clearly, ${\mathscr{R}|}_{_E} = (\mathscr{R}_I \cap \mathscr{R}_L)\cap (E\times E) = \mathscr{R}\cap (E\times E)$ is a congruence on $E$.
		
For any $a, b\in G$, we need to show that $(\lambda_a , \lambda_b)\in \mathscr{R}$. Clearly, $a \, {\leq}_q \, \, \lambda_a(x)$, $b \, {\leq}_q \, \, \lambda_b(x)$, for all $x\in G$ implies $(\lambda_a, \lambda_b) \in \mathscr{R}_L$. Also, $(\lambda_a+\lambda_{a+b})(x) = a+a'+b+b' = (\lambda_b+\lambda_{a+b})(x)$, for all $x\in G$ implies $\lambda_a+\lambda_{a+b} = \lambda_b+\lambda_{a+b}$ and thus $(\lambda_a, \lambda_b) \in \mathscr{R}_I$. Hence $(\lambda_a , \lambda_b) \in \mathscr{R}_I \cap \mathscr{R}_L = \mathscr{R}$. Since $E(G)$ contains at least two idempotents, we must have ${\mathscr{R}|}_{_E} \neq {\epsilon}_{_E}$. We now show that ${\mathscr{R}|}_{_E}$ is the monolith of $E$.
		
Let $\mathscr{S}$ be any congruence on $E$ such that $\mathscr{S} \neq {\epsilon}_{_E}$. There exists a pair $(\varphi , \psi) \in \mathscr{S}$ such that $\varphi \neq \psi$. Since $E$ is separated by idempotents, so   there exists an element $e\in E(G)$ such that $\varphi(e) \neq \psi(e)$. Here $E(G)$ is a semilattice and $\varphi(e), \psi(e) \in E(G)$. Without loss of generality, we can consider $\psi(e) \nleq \varphi(e)$. (If $\varphi(e) \nleq \psi(e)$, then we consider $(\psi, \varphi)$ instead of the pair $(\varphi, \psi)$.)
		
Let $b_1, b_2$ be two arbitrary elements of $G$ such that $b_1 \, {\leq}_q \, b_2$. Then $\mu_{_{b_1,b_2,\varphi(e)}}$ belongs to $E$. Since $\mathscr{S}$ is a congruence, we have $(\mu_{_{b_1,b_2,\varphi(e)}}\cdot \varphi \cdot \lambda_{e}, \, \, \,  \mu_{_{b_1,b_2,\varphi(e)}} \cdot \psi \cdot \lambda_{e})\in \mathscr{S}$, i.e, $(\lambda_{b_1} , \lambda_{b_2})\in \mathscr{S}$. Thus, for any two elements $b_1, b_2 \in G$ with $b_1 \, {\leq}_q \, b_2$ implies $(\lambda_{b_1} , \lambda_{b_2})\in \mathscr{S}$.
		
Let $c_{_1},c_{_2}$ be any two elements of $G$. Let $c=c_{_1}+c_{_2}$. Then it is easy to verify $c_{_1} \, {\leq}_q \, c$ and $c_{_2} \, {\leq}_q \, c$. Therefore $(\lambda_{c_1}, \lambda_{c}) \in \mathscr{S}$ and $(\lambda_{c_2} ,\lambda_{c}) \in \mathscr{S}$ and hence $(\lambda_{c_1} ,\lambda_{c_2})\in \mathscr{S}$. Therefore $(\lambda_{p} ,\lambda_{q}) \in \mathscr{S}$, for any two elements $p, q\in G$.
		
Let $(f, g)\in {\mathscr{R}|}_{_E}$ and $f \neq g$. There exist elements $u, v, r \in G$ such that $u \, {\leq}_q \, \, f(x)$, $v \, {\leq}_q \, \, g(x)$, for all $x \in G$ and $f+\lambda_{r} = g+\lambda_{r}$. Here $u \, {\leq}_q \, \, f(x)$, for all $x \in G$ implies $u+u'+f(x) = f(x)$ for all $x \in G$, i.e., $\lambda_{u}(x)+f(x) = f(x)$ for all $x\in G$, i.e., $\lambda_u+f = f$ and thus $(\lambda_{u}+ f, f) \in \mathscr{S}$. Similarly, from $v \, {\leq}_q \, \, g(x)$ for all $x \in G$ implies $(\lambda_v+g, g) \in \mathscr{S}$. Since $(\lambda_{u} ,\lambda_{r}) \in \mathscr{S}$ and $\mathscr{S}$ is a congruence, we have $(\lambda_{u}+f, \lambda_{r}+f) \in \mathscr{S}$. Therefore $(f , \lambda_{r}+f) \in \mathscr{S}$. Similarly, we can prove that $(\lambda_{r}+g, g) \in \mathscr{S}$. Thus, $f \, \, \mathscr{S} \, \, (f+\lambda_r) = (g+\lambda_r) \, \, \mathscr{S} \, \, g$ and hence $(f, g)\in \mathscr{S}$. Therefore, ${\mathscr{R}|}_{_E} \subseteq \mathscr{S}$ for any congruence $\mathscr{S}$ on $E$ such that $\mathscr{S} \neq {\epsilon}_{_E}$. Thus ${\mathscr{R}|}_{_E}$ is the monolith of $E$ and $E$ is subdirectly irreducible.
\end{proof}

\begin{cor}
Let $(G, +)$ be a commutative inverse semigroup containing at least two idempotents. If $End(G)$ is separated by idempotents, then it is subdirectly irreducible and its monolith is $\mathscr{R}$.
\end{cor}

\section{Simplicity}

For a commutative inverse semigroup $(G, +)$, the set $\textbf{E}_{E(G)} = End\Bigl(E(G)\Bigr)$ is a subsemiring of the additive inverse semiring $End(G)$ and also $\textbf{M}_{E(G)}$ is a subsemiring of 
$\textbf{M}_{G}$. 

\begin{theorem}
Let $(G, +)$ be a commutative inverse semigroup and $E$ be a subsemiring of $\textbf{E}_{E(G)}$ such that 
$\textbf{M}_{E(G)} \subseteq E$.
	
(i) If $E(G)$ has both a least and a largest element with respect to the partial order relation ${\leq}$, then $E$ is simple.	
	
(ii) If $E$ is simple and $\textbf{id}_{E(G)} \in E$, then $E(G)$ has both a least and a largest element with respect to the partial order relation ${\leq}$.
\end{theorem}

\begin{proof}
Clearly, $E$ satisfies all conditions of Theorem \ref{th:monolith}. By Theorem \ref{th:monolith}, it follows that $E$ is simple if and only if the monolith ${\mathscr{R}|}_{_E}$ of $E$ is equal to $E\times E$. If $E(G)$ has both a least and a largest element with respect to partial order relation ${\leq}$, then it is clear that ${\mathscr{R}|}_{_E} = E \times E$. If $E(G)$	has no least element then the range of $\textbf{id}_{E(G)}$ is not lower bounded subset of $E(G)$ and hence $(\textbf{id}_{E(G)}, \lambda_{e}) \notin {\mathscr{R}|}_{_E}$, for any $e \in E(G)$. If $E(G)$ has the least element $e_{_0}$ but no largest element, then for each $e \in E(G)\smallsetminus \{e_{_0}\}$, $(\textbf{id}_{E(G)}, \lambda_{e}) \notin {\mathscr{R}|}_{_E}$, for any $e \in E(G)\smallsetminus \{e_{_0}\}$.
\end{proof}

\begin{cor}
Let $(G, +)$ be a commutative inverse semigroup. Then $\textbf{E}_{E(G)}$ is simple if and only if $E(G)$ has both a least and a largest element.
\end{cor}

\begin{cor}
Let $(G, +)$ be a commutative inverse semigroup and $E(G)$ is finite. Then $\textbf{E}_{E(G)}$ is simple if and only if $E(G)$ is a lattice with respect to partial order relation ${\leq}$.	
\end{cor}

\begin{theorem}
The semiring $\textbf{M}_{G}$ is simple for any commutative inverse semigroup $G$.
\end{theorem}

\begin{proof} Let $E = \textbf{M}_{G}$. We show that $E$ is simple. For this it is enough to show that the monolith of $E$ is $E\times E$. First we show that any two $\mu$'s are ${\mathscr{R}|}_{_E}$ equivalent. For this let $\mu_{a,b,c}; \mu_{p,q,r}$ be two $\mu$'s. Then $a \, {\leq}_q \, \, \mu_{a,b,c}(x)$ for all $x \in G$ and $p \, {\leq}_q \, \, \mu_{p,q,r}(x)$ for all $x \in G$. Moreover, one can easily check that $\mu_{a,b,c}+\lambda_{a+b+p+q} = \mu_{p,q,r}+\lambda_{a+b+p+q}$. Therefore, $(\mu_{a,b,c}, \mu_{p,q,r}) \in {\mathscr{R}|}_{_E}$ and thus any two $\mu$'s are ${\mathscr{R}|}_{_E}$ equivalent. Again, by Remark \ref{rem:monolith}, it follows that product of two $\mu$'s is again another $\mu$'s. Since $E$ is generated by $\mu$'s, it follows that every element of $E$ is a sum of finitely many $\mu$'s. Since any two $\mu$'s are ${\mathscr{R}|}_{_E}$ equivalent, it follows that any two finite sums of $\mu$'s are ${\mathscr{R}|}_{_E}$ equivalent. Consequently, $E = \textbf{M}_{G}$ is simple. 
\end{proof}

\section{Simpleness of endomorphism hemiring of a commutative inverse monoid}

Let $G$ be a commutative inverse monoid with identity $0$ and $End_0(G)$ be the hemiring of all $0$ fixing endomorphisms of $G$. For every pair $a, b$ of elements of $G$, we define $\tau_{a,b} \in End_0(G)$ by 
$\tau_{a,b} = \mu_{_{0, b, a}}$. Let $T_G$ be the subsemiring of $End_0(G)$ generated by all endomorphisms of the form $\tau_{a, b}$ (where $a, b \in G$). Also, if $R_G = \{\varphi \in End_0(G) \, :$ range of $\varphi$ is a finite subset of $G\}$, then $R_G$ is a subsemiring of $End_0(G)$ with zero. Moreover, $T_G\subseteq R_G$ and $R_G$ is an ideal of $End_0(G)$.

\begin{definition} \cite{go}
An element $a$ in a semiring $S$ is said to be infinite if and only if $a + x = a = x + a$ for all $x \in S$. Infinite element in a semiring is unique and is denoted by $\infty$. 
\end{definition}

\begin{lemma} \label{le: leftcomposion}
For any $a,b,c,d \in G$ and $\varphi \in End_0(G)$, we have $\varphi \cdot \tau_{a,b}= \tau_{a,\varphi(b)}$ and 
\begin{align*}
\tau_{c,d} \cdot \varphi \cdot \tau_{a,b} = 
\begin{cases} 
\theta, &\text{if}\ \varphi(b) \, {\leq}_q \, c \\
\tau_{a,d}, &\text{otherwise}.
\end{cases} 
\end{align*}
where $\theta(x)=0$, for all $x \in G$, is the zero element of $End_0(G)$. If $(G, +, 0)$ has an absorbing element $\infty \in G$, then $\tau_{_{0,\infty}}$ is infinite element in $End_0(G)$.
\end{lemma}

\begin{proof}
Since $\varphi(0)=0$ for all $\varphi \in End_0(G)$, so 
\begin{align*} 
\varphi \cdot \tau_{a,b}(x) &= \begin{cases} 
\varphi(0), &\text{if}\ x \, {\leq}_q \, a \\
\varphi(b+b'), &\text{otherwise}
\end{cases}
&= \begin{cases} 
0, &\text{if}\ x \, {\leq}_q \, a \\
\varphi(b)+\varphi'(b), &\text{otherwise}
\end{cases}
&= \tau_{a,\varphi(b)}(x) &\text{for all $x \in G.$}
\end{align*}

Applying this formula we get
\begin{align*}
\tau_{c,d} \cdot \varphi \cdot \tau_{a,b} = \tau_{c,d} \cdot \tau_{a,\varphi(b)} =\tau_{a,\tau_{c,d}(\varphi(b))} 
= \begin{cases} 
\theta, &\text{if}\ \varphi(b) \, \, {\leq}_q \, c \\
\tau_{a,d}, &\text{otherwise}
.
\end{cases}
\end{align*}


Now,	
\begin{align*}
\tau_{_{0,\infty}}(x)=
\begin{cases}
0, &\text{if}\ x \, {\leq}_q \, 0 \\
\infty, &\text{otherwise}
\end{cases}
= \begin{cases}
0, &\text{if}\ x = 0 \\
\infty, &\text{otherwise}
\end{cases}
 \hspace{1.5em} \text{for all $x \in G$.}
\end{align*}
For any $\psi \in End_0(G)$, $x\in G\setminus \{0\}$, we have $(\psi + \tau_{_{0,\infty}})(x) = \psi(x) + \infty = \infty$, so that $\psi + \tau_{_{0,\infty}} = \tau_{_{0,\infty}}$. Similarly, we can show that 
$\tau_{_{0,\infty}}+\psi = \tau_{_{0,\infty}}$. Therefore, $\tau_{_{0,\infty}}$ is infinite element in 
$End_0(G)$.
 \end{proof}

\begin{lemma}
Let $(G, +, 0)$ be a commutative inverse monoid such that the set $ \{ x \in G \; : \; \varphi(x) \, \, {\leq}_q \, a\}$ contains finite number of elements of $G$ for any $a \in G$ and for any $\varphi \in End_0(G)$. Then $T_G$ is an ideal of $End_0(G)$. In particular, if $G$ is finite, then $T_G$ is an ideal of $End_0(G)$.
\end{lemma}

\begin{proof}
Using Lemma \ref{le: leftcomposion}, we get that $T_G$ is a left ideal of $End_0(G)$. Now for $a,b,x \in G$ and $\varphi \in End_0(G)$,
\begin{align*}
\tau_{a,b}(\varphi(x)) = \begin{cases} 
0, &\text{if}\ \varphi(x) \, \, {\leq}_q \, a \\
b+b', &\text{otherwise}.
\end{cases}
\end{align*}
Since $ \{ x \in G \; : \; \varphi(x) \, \, {\leq}_q \, a \}$ contains finite number of elements of $G$, for any $a \in G$ and $\varphi \in End_0(G)$, so we can consider the set $\{y_1, y_2, \ldots, y_n\}$ = $ \{ x \in G \; : \; \varphi(x) \, \, {\leq}_q \, a\}$. Let $g= y_1+y_2+\cdots+y_n$. Then $y_i+y_i'+g = y_i+y_i'+y_1+y_2+\cdots+y_n=y_1+y_2+\cdots+y_n = g$ for all $y_i\in \{y_1, y_2, \ldots, y_n\}$ implies $y_i \, \, {\leq}_q \, g$ for each $y_i \in \{y_1, y_2, \ldots, y_n\}$. This implies $\{ x \in G \; | \; \varphi(x) \, \, {\leq}_q \, \, a  \} \subseteq  \{ x \in G \; | \; x \, \,  {\leq}_q \, \, g  \}$. For the reverse inclusion, let $z \in \{ x \in G \; | \; x \, \, {\leq}_q \, \, g  \}$. Then $z+z'+g = g$. Now, $\varphi(y_i)+\varphi'(y_i)+a=a$, for each $y_i\in \{y_1, y_2, \ldots, y_n\}$ implies $\varphi(y_i)+\varphi(y_i')+a+a'=a+a'$, for each $y_i\in \{y_1, y_2, \ldots, y_n\}$ and hence $\varphi(g)+\varphi(g')+a+a'=a+a'$. Now, $z+z'+g = g$ implies $\varphi(z)+\varphi(z')+\varphi(g)=\varphi(g)$. This implies $\varphi(z)+\varphi(z')+\varphi(g)+\varphi(g')+a+a'=\varphi(g)+\varphi(g')+a+a'=a+a'$, i.e., $\varphi(z)+\varphi(z')+a+a'=a+a'$, i.e, $\varphi(z)+\Big(\varphi(z)\Big)'+a = a$, which gives $\varphi(z) \, \, {\leq}_q \, \, a$ and hence $\{ x \in G \; | \; x \, \, {\leq}_q \, \, g  \} \subseteq \{ x \in G \; | \; \varphi(x) \, \, {\leq}_q \, \, a \}$. Therefore, $\{ x \in G \; | \; x \, \, {\leq}_q \, \, g  \} = \{ x \in G \; | \; \varphi(x) \, \, {\leq}_q \, \, a\}$.  Hence for all $x\in G$ 
\begin{align*}
\tau_{a,b}(\varphi(x)) = \begin{cases} 
0, &\text{if}\ \varphi(x) \, \, {\leq}_q \, a \\
b+b', &\text{otherwise}
\end{cases} 
\, \, \, \, = \begin{cases} 
0, &\text{if}\ x \, {\leq}_q \, g \\
b+b', &\text{otherwise}
\end{cases} \, \, \, \, = \tau_{g,b}(x).
\end{align*}
Hence $T_G$ is an ideal of $End_0(G)$.
\end{proof}

\begin{theorem}\label{th:finite}
Let $(G, +, 0)$ be a commutative inverse monoid and $S$ be a subsemiring of  $End_0(G)$ containing $\textbf{id}_G$. If $S$ is ideal-simple and $S\cap R_G \neq \{\theta\}$, then $G$ is finite.	
\end{theorem}

\begin{proof}
Since $S \cap R_G \neq \{\theta\}$, so we take a nonzero endomorphism $\beta \in S\cap R_G$ and consider the ideal $\langle\beta\rangle$ of $S$ generated by $\beta$. Since $S$ is ideal-simple and contains $\textbf{id}_G$, so there exist some endomorphisms $\alpha_1,\gamma_1,\alpha_2,\gamma_2,\cdots,\alpha_n,\gamma_n \in S$ such that $\textbf{id}_G = \alpha_1 \cdot\beta \cdot\gamma_1+\alpha_2 \cdot\beta \cdot\gamma_2+\cdots+\alpha_n \cdot\beta \cdot\gamma_n$. Since $\beta \in R_G$, so we get $\textbf{id}_G\in R_G$ and hence $G$ is finite.	
\end{proof}

\begin{cor}
Let $(G, +, 0)$ be a commutative inverse monoid and $S$ be a subsemiring of $End_0(G)$ containing $\textbf{id}_G$. If $S$ is ideal-simple and $S\cap T_G \neq \{\theta\}$, then $G$ is finite.
\end{cor}

\begin{proof}
Since $T_G \subseteq R_G$, so by Theorem \ref{th:finite}, it follows that $G$ is finite.
\end{proof}

\begin{theorem}
Let $(G, +, 0)$ be a commutative inverse monoid containing at least two idempotents and an absorbing element 
$\infty$. Let $E$ be a subsemiring of $End_0(G)$ such that $E$ is separated by idempotents and $\textbf{T}_G \subseteq E$. Then $E$ is congruence-simple.
\end{theorem}

\begin{proof}
Let $\Re$ be a semiring congruence on $E$ such that $\Re \neq {\Delta}_E$. Then there exists $\varphi , \psi \in E$ with $\varphi \neq \psi$, but $\varphi \, \, \Re \, \, \psi$. Since $E$ is separated by idempotents, so there exists $e \in E(G)$ such that $\varphi(e) \neq \psi(e)$. Here $E(G)$ is a semilattice and $\varphi(e), \psi(e) \in E(G)$. Without any loss of generality, we may consider $\varphi(e) \nleq \psi(e)$ [If $\varphi(e) \nleq \psi(e)$, then we consider $(\psi, \varphi)$ instead of the pair $(\varphi, \psi)$]. 
	
For all $a,b \in G$; we have $\tau_{_{a,e}}, \tau_{_{\psi(e),b}} \in E$. Using Lemma \ref{le: leftcomposion}, we get $\tau_{_{\psi(e),b}} \cdot \varphi \cdot \tau_{_{a,e}} = \tau_{_{a,\tau_{\psi(e),b}(\varphi(e))}} = \tau_{_{a,b}}$, because $\varphi(e) \nleq \psi(e)$. Also Using Lemma \ref{le: leftcomposion}, we get $\tau_{_{\psi(e),b}} \cdot \psi \cdot \tau_{_{a,e}} = \tau_{_{a,\tau_{\psi(e),b}(\psi(e))}} = \theta$. Since $\Re $ is a semiring congruence on $E$, so $(\tau_{_{\psi(e),b}} \cdot \varphi \cdot \tau_{_{a,e}}) \, \, \Re \, \, (\tau_{_{\psi(e),b}} \cdot \psi \cdot \tau_{_{a,e}})$ and hence $\tau_{_{a,b}} \, \, \Re \, \, \theta$. Taking $a=0$ and $b=\infty$, we have $\tau_{_{0,\infty}} \, \,\Re \, \, \theta$ and hence $\tau_{_{0,\infty}} = \alpha + \tau_{_{0,\infty}} \, \, \Re \, \, \alpha + \theta= \alpha$, for all $\alpha \in E$, since $\Re $ is a semiring congruence on $E$. Therefore, $\Re = E\times E $ and consequently, $E$ is congruence-simple.
\end{proof}

\begin{cor}
Let $(G, +, 0)$ be a commutative inverse monoid containing at least two idempotents and an absorbing element 
$\infty$. If $End_0(G)$ is separated by idempotents, then it is congruence-simple.
\end{cor}

\end{document}